\documentclass{article}
\usepackage[utf8]{inputenc}
\usepackage{amsmath, amsthm, amssymb, amsfonts}
\usepackage{bm}
\usepackage{dsfont}
\usepackage[utf8]{inputenc}
\usepackage{rotate}
\usepackage{tikz}
\usepackage{tikz-cd}
\usepackage[arrow,matrix,curve]{xy} 	
\usepackage{xcolor}
\usepackage{todonotes}
\usepackage{alltt} 
\usepackage{calc}
\usepackage{comment}

\theoremstyle{plain}
\newtheorem{theorem}{Theorem}[section]
\newtheorem{proposition}[theorem]{Proposition}
\newtheorem{lemma}[theorem]{Lemma}

\theoremstyle{definition}

\newtheorem{remark}[theorem]{Remark}

\newcommand{\norm}[1]{\left\lVert#1\right\rVert}

\title{An Explicit non-Poissonian Pair Correlation Function}
\author{Christian Wei\ss{}}
\date{\today}
\begin{document}

\maketitle

\begin{abstract} A generic uniformly distributed random sequence on the unit interval has Poissonian pair correlations. Usually, the pair correlations statistic is therefore studied for equidistributed sequences. At the same time, there are only very few explicitly known examples of sequences with this property and many types of deterministic sequences have been proven to fail having the Poissonian pair correlation property. In this paper we study the pair correlation statistic in the non-uniform case and analyze the first elementary example of such a sequence, namely $x_n := \left\{ \frac{\log(2n-1)}{\log(2)} \right\}$, which is a standard low-dispersion sequence. The proof heavily relies on a full understanding of the gap structure of $(x_n)_{n=1}^N$. Furthermore, we discuss differences to the weak pair correlation function which turns out to be linear.
\end{abstract}

\section{Introduction}

There are different concepts, which quantify the degree of uniformity of a sequence $(x_n)_{n \in \mathbb{N}}$ in $[0,1]$. One natural object to study is the behavior of gaps between the first $N \in \mathbb{N}$ elements of $(x_n)_{n \in \mathbb{N}}$ on a local scale. This idea was popularized by Rudnick and Sarnak in \cite{RS98} and is formalized on the $N$-point pair correlation function defined by
$$F_{N}(s) := \frac{1}{N} \# \left\{ 1 \leq k \neq l \leq N \ : \ \left\| x_k - x_l\right\| \leq \frac{s}{N} \right\},$$
where $\left\| \cdot \right\|$ is the distance of a number from its nearest integer. A sequence $(x_n)_{n \in \mathbb{N}}$ is said to have Poissonian pair correlations if the limit
$$\textcolor{black}{\lim_{N \to \infty}} F_{N}(s) =: F(s)$$
exists and satisfies $F(s) = 2s$ for all $s \geq 0$. Any sequence which possesses Poissonian pair correlations, is uniformly distributed according to \cite{ALP18, GL17, Ste18}. The task to calculate the pair correlation statistic of a given sequence is closely connected to studying the gap structure of the sequence, i.e. the distances between two neighboring elements of the sequence, although there the two concepts are not equivalent, see \cite{AHZ25}. This connection of gaps and the pair correlation statistic partially explains the current popularity of the topic and has recently been very explicitly addressed in several publications, see e.g. \cite{LS20}, \cite{Wei22a} and \cite{Wei23a}. But also independently of the more specific pair correlation statistic, the gap structure and more generally the nearest neighbor graph of sequences in $[0,1]^d$ is extensively explored in the literature, see e.g. \cite{Don09, EM04, PSZ16, Tah17} and \cite{HM20}.\\[12pt]
Although a uniformly distributed random sequence in $[0,1]$ generically has Poissonian pair correlations (see e.g. \cite{Mar07} for a proof), there are only few explicitly known such examples, the most prominent one being $x_n = \left\{ \sqrt{n} \right\}$ for $n$ not a perfect square, see \cite{BMV15}, and more recent ones in {\cite{LT22,LST21,RS24,Hau25}. On the contrary, many canonical candidates of uniformly distributed sequences are known to fail having Poissonian pair correlations, see e.g., \cite{BCC19, LS20, PS19, WS22}. \\[12pt]
Recall that the gaps of a finite sequence $(x_n)_{n=1}^N$ are defined as follows: suppose that the elements of the sequence are ordered by size, i.e. $x_1 \leq x_2 \leq \ldots \leq x_N$. Then the gaps $g_i$ are
$$g_i := x_{i+1} - {x_i}, \quad i=1,\ldots,N-1, \qquad \textrm{and} \qquad g_N = x_1 + 1 - x_N,$$
that is the distance between two neighboring elements of $(x_n)_{n=1}^N$ when considered as elements on the torus $\mathds{T}_1$.\\[12pt]
Much simpler than understanding the entire gap structure of a sequence is to consider only the largest gap, whose length is called the dispersion of the sequence, $\textrm{disp}_N(x_n) = \max_{i=1,\ldots,N} g_i$. According to \cite{Nie92} (Chapter 6.2), the best possible decay rate is $\textrm{disp}_N(x_n) \leq c N^{-1}$ as $N \to \infty$ with $c \in \mathbb{R}$ independent of $N$. A sequence with this behavior is also called a low-dispersion sequence. According to \cite{Nie84}, the best possible convergence $\lim_{N \to \infty} \textrm{disp}_N(x_n) N = \frac{2}{\log(4)}$ is obtained for the sequence
$$x_n := \left\{ \frac{\log(2n-1)}{\log(2)} \right\} = \left\{ \log_2(2n-1)\right\}, \quad n = 1,2,\ldots,$$
where $\left\{ \cdot \right\}$ denotes the fractional part of a real number. The proof is short and obviously only relies on the knowledge of the largest gap. It is remarkable that at the same time $(x_n)_{n \in \mathbb{N}}$ is not a low-discrepancy sequence -- for a definition see again \cite{Nie92}, Chapter 3 -- and not even uniformly distributed. However, much more effort is required to calculate the limiting function of the pair correlation statistic, which is the main result of this paper. 
\begin{theorem} \label{main:thm} Let 
$$x_n := \left\{ \frac{\log(2n-1)}{\log(2)} \right\}.$$
The limiting function $F(s)$ of the finite pair correlation statistic $F_N(s)$ as $N \to \infty$ exists and is given by
\[ 
F(s) = \begin{cases}  
0 & 0 \leq s < \log(4)^{-1}\\
3k - \frac{3k^2+2k}{4\log(2)s} & k \log(4)^{-1} \leq s < (k+1) \log(4)^{-1}, k \in \mathbb{N} \ \textrm{even},\\ 3k + 1 - \frac{3k^2+4k+1}{4\log(2)s} & k \log(4)^{-1} \leq s < (k+1) \log(4)^{-1}, k \in \mathbb{N} \ \textrm{odd}.\end{cases}
\]
\end{theorem}
\begin{remark} Note that $F(s)$ behaves asymptotically like $2s \frac{3}{4}\log(4)(1+o(1))$. This fact will be further discussed at the end of Section~\ref{sec:proofs}, where we consider weak pair correlations.
\end{remark}
\begin{center}
    \begin{figure}
        \includegraphics[scale=0.6]{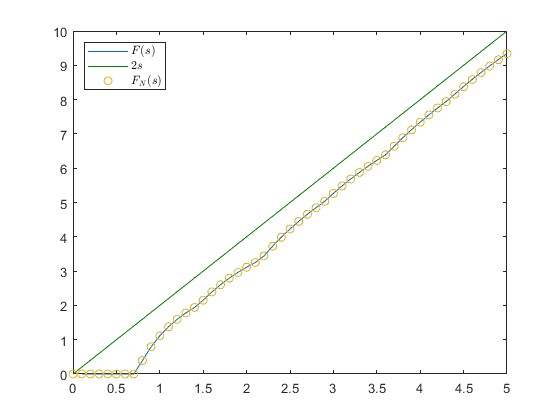}
        \caption{Comparison of $F_N(s)$ and $F(s)$ for $N=1,000$.}
        \label{fig:PCS}
    \end{figure}
\end{center}
Figure~\ref{fig:PCS} shows a comparison of $F_N(s)$ and the limit function $F(s)$ for $\textcolor{black}{N=1,000}$, which indicates that the speed of convergence of $F_N(s)$ to the function $F(s)$ in Theorem~\ref{main:thm} is relatively fast. Indeed, Theorem~\ref{thm:bounds} makes this statement more precise. Moreover, from Theorem~\ref{main:thm} we see that $F(s)$ is a continuous function.\\[12pt]
As indicated above, the main ingredient for deriving the function $F(s)$ lies in a complete understanding of the gap structure of $(x_n)_{n=1}^N$ including their explicit ordering, see Proposition~\ref{prop:gaps}. This allows us to infer information not only about the neighbor of each $x_i^*$ in the ordered sequence, but also about the distances of $x_i^*$ from all elements $x_{i+k}^*$ for $k \leq N$ (Proposition~\ref{prop:kmax}, Lemma~\ref{lem:bounds:J}). While the techniques applied here are entirely elementary, the proof might nonetheless serve as a blueprint for the calculation of other limiting functions if the gap structure is known (this might, for instance, be within reach for van der Corput and Kronecker sequences).\\[12pt]
The main difficulty in calculating the function $F(s)$ is to describe the gap structure of a sequence and all successful attempts to solve this problem prior to the present paper are based on very non-elementary arguments: For instance in \cite{MS13}, it is shown by using arguments originally developed in the context of spectral statistics of quantum graphs that the gap distribution of $(x_n)_{n \in \mathbb{N}} = (\left\{\log(n)\right\})_{n \in \mathbb{N}}$ has an explicit distribution which is not Poissonian. Although a precise expression for the number of neighbors with distance $<s/N$ is given, no formula for the distance of the second, third, and so on nearest neighbor is derived therein. However, this would be required to calculate the limit $F(s)$ from Theorem~\ref{main:thm}, while the distribution function on its own does not suffice for this purpose.\\[12pt] 
On the contrary, the authors prove in \cite{LT22} that sequences of the form $\left\{\alpha \log(n)^A\right\}$ for $A > 1$ and $\alpha > 0$ are Poissonian. Another recent result in \cite{Lut20} uses arguments from random matrix theory to calculate the limiting pair correlation function of orbits of a point in hyperbolic space under the action of a discrete subgroup.

\paragraph{Acknowledgements.} The research on this article was conducted during a stay at the Universit\'{e} de Montr\'{e}al, whom the author would like to thank for their hospitality. Furthermore, he thanks Chris Lutsko for useful comments on an earlier version of this paper.
\section{Proofs of Results} \label{sec:proofs}
\begin{proposition} \label{prop:gaps} Let $N \in \mathbb{N}$ with (unique) $2$-adic expansion $N = \sum_{l = 0}^L a_l 2^l$, where $a_k \in \{0,1\}$ for all $l=0,\ldots,L$ and $a_L=1$. Furthermore let $N_0 = 2 \cdot (N-2^L)$. Then the gaps of $(x_n)_{n=1}^N$ are
$$
g_i = \begin{cases} \tfrac{1}{\log(2)} \left( \log(2^{L+1}+i)-\log(2^{L+1}+i-1) \right)  & 1 \leq i \leq N_0,\\ \frac{1}{\log(2)} \left( \log(N-N_0 +i) - \log(N-N_0 + i -1) \right) & N_0 + 1 \leq i \leq N.\end{cases}
$$
\end{proposition}
The knowledge about the entire gap structure will ultimately put us in the position to calculate the limiting pair correlation function $F(s)$ of $(x_n)$.
\begin{proof} For $N=2$, there are only two gap lengths, namely $(\log(3)-\log(2))/\log(2)$ and $(\log(4)-\log(3))/\log(2)$. Next note for $n >2$ that
\[
x_n = \frac{\log(z_n)}{\log(2)}
\]
with
\[
z_n = 1 + \frac{2k-1}{2^{m}}
\]
if we write $n-1 = 2^{m-1}+k$ with $k \leq 2^{m-1}-1$. For $n=1,\ldots,2^m$, the expression $z_n$ therefore takes all values $z_n= 1+ \frac{l}{2^{m+1}}$ with \textbf{even} \textcolor{black}{numerator} $l$ (note that the exponent in the denominator is $2^{m+1}$ instead of $2^m$). For the next $n=2^{m}+1,\ldots,2^{m+1}$, we have $z_n = 1 + \frac{k}{2^{m+1}}$ for all odd natural numbers $k$ in ascending order. Thus, it follows for each $n$ in this range that $x_n$ splits the longest existing gap of $(x_j)_{j=1}^{n-1}$ (by the concavity of the logarithm), which has length $(\log(n)-\log(n-1))/\log(2)$, into one of length $(\log(2n)-\log(2n-1))/\log(2)$ and one of length $(\log(2n-1)-\log(2n-2))/\log(2)$. 
\end{proof}
Proposition~\ref{prop:gaps} shows that the pair correlation function for this specific sequence is invariant under shifts $x_n \mapsto x_n + C \mod 1$ with $C \in \mathbb{R}$. In the following, we may therefore (for fixed $N$) replace in our specific situation $(x_n)_{n=1}^N$ by the sequence with ascending gap lengths
\begin{align} \label{eq:gaps} g_i' = \frac{1}{\log(2)} \Big( \log(2N - (i-1)) - \log(2N-1 - (i-1)) \Big), \quad i = 1,\ldots,N.\end{align}
In other words, we may assume
$$y_1 = 0, \quad y_{i+1} = y_i + g_i', \quad i = 1,\ldots,N-1.$$ 
Then, the distance $\norm{y_N-y_1}$ is $\tfrac{1}{\log(2)}(\log(N+1)-\log(N))$, as desired. For fixed $1 \leq m \leq N$ and $s \geq 0$ we now turn to the calculation of
$$K_{m}(N,s) := \# \left\{ n \, : \, 1 \leq  m < n  \leq N + m, \norm{y_m-y_{i \, (\textrm{mod} N)}} \leq \frac{s}{N}\, \textrm{for all} \, m < i \leq n \right\} $$
and set $K_{\max}(N,s) = \max_{1 \leq m \leq N} K_{m}(N,s),$ i.e., $K_{\max}(N,s)$ is the maximum number of points with distance at most $\frac{s}{N}$, which lie to the right of any given $y_m$. Similarly, we define $K_{\min}(N,s) = \min_{1 \leq m \leq N} K_{m}(N,s)$ as the minimum number of points with distance at most $\frac{s}{N}$, which lie right of any given $y_m$. (Note that the same then also holds true for $x_m$). 
\begin{proposition} \label{prop:kmax} For $s/N < \frac{1}{2}$\textcolor{black}{,} let 
$$c_{N,s} =  2^{s/N} -1.$$
Then we have
$$K_{\max}(N,s) = \left\lfloor \frac{2Nc_{N,s}}{1+c_{N,s}} \right\rfloor$$
and
$$K_{\min}(N,s) = \left\lfloor N c_{N,s}\right\rfloor.$$
\end{proposition}
\begin{proof} It follows immediately from the gap structure of the $y_n$, compare \eqref{eq:gaps}, that $K_{\max}(N,s)$ is obtained for the elements right of $y_1$, i.e., when considering the successive gaps $g_1',g_2',\ldots,g_{K_{\max}(N,s)}'$. The distance between $y_1$ and $y_{k+1}$ for $k \neq N$ is given by $\tfrac{1}{\log(2)}(\log(2N)-\log(2N-k))$. Using the defining condition on the norm implies
$$\frac{1}{\log(2)} \left( \log(2N)-\log(2N-k) \right) \leq \frac{s}{N},$$
which is equivalent to
$$\frac{2N}{2N-k} \leq 1+c_{N,s}$$
and
$$k \leq \frac{2Nc_{N,s}}{1+ c_{N,s}}.$$
Since $k$ is an integer, the first claim follows. The minimum is obviously attained if the right-most point fulfilling the condition of $K_m(N,s)$ is $y_N$, because the gap sizes are ordered in ascending order. Hence, it must hold that
\[
\frac{1}{\log(2)}\left( \log(N+k)-\log(N) \right) \leq \frac{s}{N}
\]
and the calculation thus works similarly as for $K_{\max}(N,s)$.
\end{proof}
Moreover, the limit $\lim_{N \to \infty} K_{\max}(N,s)$ will appear several times in the remainder of the paper and we therefore turn to its calculation now.
\begin{lemma} \label{lem:kmax} For fixed $s \geq 0$ the function $K_{\max}(N,s)$ is bounded. It furthermore holds that
$$\lim_{N \to \infty} K_{\max}(N,s) = \begin{cases} \left\lfloor \log(4)s \right\rfloor & \log(4)s \notin \mathbb{Z} \\ \log(4)s-1 & \log(4)s \in \mathbb{Z} \end{cases}$$
and
$$\lim_{N \to \infty} K_{\min}(N,s) = \left\lfloor \log(2)s \right\rfloor.$$
\end{lemma}
\begin{proof} When calculating the limit of $K_{\max}(N,s)$ assume first that $\log(4)s \notin \mathbb{Z}$.  Looking at the reciprocal expression
$$\frac{1+c_{N,s}}{2c_{N,s}N} = \frac{1}{2c_{N,s}N} +\frac{1}{2N},$$
we see that it suffices to consider the limit of the first term.  
Replacing $c_{N,s}$ by its definition yields
$$\frac{1}{2N  (\exp \left(\log(2) \cdot \frac{s}{N} \right) -1)},$$
which converges to $\frac{1}{\log(4)s}$ as $N \to \infty$ as can be seen from the exponential series. If $\log(4)s\in \mathbb{Z}$, then note that 
\[
\frac{2Nc_{N,s}}{1+c_{N,s}}
\]
converges to $\log(4)s$ from below. As the floor operator is not left-continuous, we thus need to subtract $1$ in the limit.\\
For the calculation of the limit $K_{\min}(N,s)$ note that $Nc_{N,s}$ converges to $\log(2)s$ from above and that the floor operator is right-continuous.
\end{proof}
Next, we consider the problem from a different perspective and estimate how many $y_i$ have a $k$-th neighbor to their left, i.e. $y_{i - k}$, such that the distance between them is at most $\frac{s}{N}$.

\begin{lemma}  \label{lem:bounds:J} For $i \pm k \notin [1,N]$, define $y_{i \pm k} := y_{(i \pm k) \mod N}$. Let $k \leq K_{\max}(N,s)$ and set
$$J_{N,s,k} := \# \left\{ 1 \leq i  \leq N \, : \, \norm{y_i-y_{i-k}} \leq \frac{s}{N} \right\}.$$
Then we have
    $$\min \left( N,2N -1 -k - \frac{k(c_{N,s}+1)}{c_{N,s}}\right) \leq J_{N,s,k} \leq \min \left( N,2N -1 -\frac{k(c_{N,s}+1)}{c_{N,s}} \right) .$$
\end{lemma}
\begin{proof} For $k +1 \leq i \leq N$, consider 
\begin{align} \label{eq:cond2}
\norm{y_{i+k}-y_i} = \frac{1}{\log(2)} \left( \log(2N-1-i-k)-\log(2N-1-i) \right)\textcolor{black}{.}
\end{align}
The condition $\norm{y_{i}-y_{i-k}} \leq \frac{s}{N}$ is in this range thus equivalent to
\[
\frac{2N-1-i}{2N-1-i-k} \leq c_{N,s}+1,
\]
which leads to
\begin{align*} \label{eq:condition}
i \leq 2N+1 - \frac{k(c_{N,s}+1)}{c_{N,s}}.
\end{align*}
Hence there are \textit{at most}
\[
2N -1 -\frac{k(c_{N,s}+1)}{c_{N,s}}
\]
elements in $J_{N,s,k}$, because \eqref{eq:cond2} does not hold for $i \leq k$ and we thus do not know if these $i \leq k$ are in $J_{N,s,k}$ or not. Hence, we only obtain an upper bound. For the lower bound we need to take into account that we might have overestimated the number of elements by $k$.
\end{proof}

The preparatory work suffices to prove the following theorem which presents a lower and an upper bound of distance 
$$\frac{4(K_{\max}(N,s)(K_{\max}(N,s)+1) - K_{\min}(N,s)(K_{\min}(N,s)+1))}{N}$$ 
for $F_N(s)$. Note that the numerator is itself a uniformly bounded expression by Lemma~\ref{lem:kmax}. 
\begin{theorem} \label{thm:bounds} For all $N \in \mathbb{N}$ and $s \geq 0$, denote 
\begin{align} \label{eq3}
\tilde{K}(N,s) := K_{\max}(N,s)(K_{\max}(N,s)+1) - K_{\min}(N,s)(K_{\min}(N,s)+1).
\end{align}
Then, the following bounds for $F_N(s)$ hold
\begin{align*}
     4 \left(1-\frac{1}{2N}\right) K_{\max}(N,s) & - 2\left(1-\frac{1}{N}\right) K_{\min}(N,s) -  \frac{c_{N,s} + 1}{c_{N,s}N}  \cdot \tilde{K}(N,s) - \frac{\tilde{K}(N,s)}{N} \\ & \leq F_N(s) \leq  \\
    4 \left(1-\frac{1}{2N}\right) K_{\max}(N,s) & -  2\left(1-\frac{1}{N}\right) K_{\min}(N,s)  -  \frac{c_{N,s} + 1}{c_{N,s}N}  \cdot \tilde{K}(N,s).
\end{align*}
\end{theorem}
\begin{proof} The function $F_N(s)$ can be rewritten as
\begin{align*}
    F_N(s) = \frac{1}{N} & \sum_{k=1}^{K_{\max}(N,s)} \sum_{n=1}^N  \left( \# \left\{ 1 \leq  n  \leq N \, : \norm{y_n-y_{n-k}} \leq \frac{s}{N} \right\} \right.\\
    & \left. + \# \left\{ 1 \leq  n  \leq N \, : \norm{y_n-y_{n+k}} \leq \frac{s}{N} \right\} \right).
\end{align*}
Since the number of points whose $k$-th neighbor to their left lies within a distance of at most $\frac{s}{N}$ is equal to the number of points whose $k$-th neighbor to their right satisfies the same condition, we obtain
\begin{align*}
    F_N(s) = \frac{1}{N} & \sum_{k=1}^{K_{\max}(N,s)} \sum_{n=1}^N  \left( 2 \# \left\{ 1 \leq  n  \leq N \, : \norm{y_n-y_{n-k}} \leq \frac{s}{N} \right\} \right).
\end{align*}
For $k \leq K_{\min}(N,s)$ the inner sum (over $n$) is equal to $2N$. The total contribution of this part of the sum to $F_N(s)$ is thus $2K_{\min}(N,s)N$. For $k > K_{\min}(N,s)$, the bounds from Lemma~\ref{lem:bounds:J} can be applied and contribute to the upper bound of $F_N(s)$ in total
$$\frac{1}{N} \left( 2(2N-1) (K_{\max}(N,s) - K_{\min}(N,s)) - \sum_{k=K_{\min}(N,s)+1}^{K_{\max}(N,s)} 2k  \frac{c_{N,s} + 1}{c_{N,s}} \right).$$
According to the lower bound in Lemma~\ref{lem:bounds:J}, there is an additional term
$$-\sum_{k=K_{\min}(N,s)+1}^{K_{\max}(N,s)} 2k$$
which needs to be taken into account to bound $F_N(s)$ from below. Reordering the terms, finishes the proof. \end{proof}
This puts us in the position to easily prove the main theorem of this paper.
\begin{proof}[Proof of Theorem~\ref{main:thm}] As $N \to \infty$, the lower and upper bound in Theorem~\ref{thm:bounds} obviously converge to the same value (if the limit exists) by Lemma~\ref{lem:kmax}. Applying Lemma~\ref{lem:kmax} to either of the bounds yields for 
\begin{align} \label{eq2}
\begin{split}
F(s) = 4& \lfloor \log(4)s \rfloor - 2 \lfloor \log(2)s \rfloor\\
&  - \frac{\lfloor \log(4)s \rfloor(\lfloor \log(4)s \rfloor+1) - \lfloor \log(2)s \rfloor (\lfloor \log(2)s \rfloor + 1)}{\log(2)s}
\end{split}
\end{align}
(surprisingly, this limit does not depend on the question whether $\log(4)s \in \mathbb{Z}$ or not). Finally, consider $s \geq 0$ with $k \log(4)^{-1} \leq s < (k+1) \log(4)^{-1}$ for $k \in \mathbb{N}$. Then we note that $\lfloor \log(2)s \rfloor = \frac{k}{2}$, if $k$ is even, and $\lfloor \log(2)s \rfloor = \frac{k-1}{2}$, if $k$ is odd. Inserting this into \eqref{eq2} yields the claim.
\end{proof}

\paragraph{Weak Pair Correlations.} We end this paper by a short discussion of the weak pair correlation function, which is also called $\alpha$-pair correlation function. This notion was originally introduced in \cite{NP07} and is for $0 \leq \alpha \leq 1$ defined by 
$$F_{N}^\alpha (s) := \frac{1}{N^{2-\alpha}} \# \left\{ 1 \leq k \neq l \leq N \ : \ \left\| x_k - x_l\right\| \leq \frac{s}{N^\alpha} \right\},$$
so that the case $\alpha = 1$ corresponds to the usual pair correlation statistic, see also \cite{Wei22a} for an even more general notion. A sequence is said to have weak correlations if $\lim_{N \to \infty} F_{N}^\alpha (s) = 2s$ for all $s \geq 0$. Without too many changes, the proof of Theorem~\ref{thm:bounds} can be amended to the situation of $\alpha$-pair correlations. In the proof of Proposition~\ref{prop:kmax}, the variable $c_{N,s}$ needs to be replaced by
$$C_{N,s,\alpha} = 2^{s/N^\alpha} -1$$
and accordingly also $K_{\max}(N,s,\alpha)$ and $K_{\min}(N,s,\alpha)$ depend on $\alpha$. A similar remark holds for Lemma~\ref{lem:bounds:J}. For instance, the lower bound in Theorem~\ref{thm:bounds} then becomes
$$g(N,s,\alpha) = 4 \frac{K_{\max}(N,s,\alpha)}{N^{1-\alpha}} - 2 \frac{K_{\min}(N,s,\alpha)}{N^{1-\alpha}} - \frac{2}{N^{2-\alpha}} - \tilde{K}(N,s,\alpha) \cdot \frac{C_{N,s,\alpha} + 1}{C_{N,s,\alpha} N^{2-\alpha}},$$
where $\tilde{K}(N,s,\alpha)$ is defined as in \eqref{eq3} but including the changes indicated above. For any $\alpha < 1$, elementary calculations imply that
$$\lim_{N \to \infty} g(N,s,\alpha) = 2s \cdot \frac{3}{4} \log(4),$$
which is independent of $\alpha$ (and only the speed of convergence is the faster, the smaller $\alpha$ is). Thus, we see that the weak pair correlation statistic behaves significantly differently from the one with $\alpha = 1$. Nonetheless, the sequence \textcolor{black}{thus does} not have weak correlations for any $0 \leq \alpha \leq 1$ although $ \frac{3}{4} \log(4) \approx 1.0397$ is close to $1$. This is in accordance with known results, see e.g. \cite{Ste18, HZ24}.

\bibliographystyle{alpha}
\bibdata{literatur}
\bibliography{literatur}

\textsc{Ruhr West University of Applied Sciences, Duisburger Str. 100, D-45479 M\"ulheim an der Ruhr,} \texttt{christian.weiss@hs-ruhrwest.de}

\end{document}